\documentclass{amsproc}

\usepackage{amsthm, amssymb, amsmath, hyperref, enumerate}

\newtheorem{theorem}{Theorem}[section]
\newtheorem{lemma}[theorem]{Lemma}

\theoremstyle{definition}

\newtheorem{conjecture}[theorem]{Conjecture}

\theoremstyle{remark}
\newtheorem{remark}[theorem]{Remark}

\numberwithin{equation}{section}

\newcommand{\QQ}{\mathbb{Q}}
\newcommand{\PP}{\mathbb{P}}
\newcommand{\FF}{\mathbb{F}}

\newcommand{\ZZ}{\mathbb{Z}}
\newcommand{\ord}{\mathrm{ord}}
\newcommand{\gp}{\mathfrak{p}}
\newcommand{\M}[1]{M[#1]}

\begin{document}
	
	\title{Explicit supersingular cyclic curves}
	
	\author{Marco Streng}
	\address{Universiteit Leiden}
	\curraddr{}
	\email{streng@math.leidenuniv.nl}
	\thanks{The author thanks Edgar Costa, Jeroen Sijsling, and Anna Somoza for helpful discussions.}

	
	
	\date{}
	
	\dedicatory{}
	
	\begin{abstract}
Li, Mantovan, Pries, and Tang proved the existence of
supersingular curves over $\overline{\FF_p}$ in each of the
special families of curves in Moonen's
classification.
Their proof does not provide defining equations of these curves.
We make some of their results explicit using the reductions
modulo~$p$
of previously computed curves with complex multiplication.
\end{abstract}
	
	\maketitle

	\section{Introduction}

	Moonen~\cite[Table~1]{moonen-special} showed that there are exactly $20$ families
	of cyclic covers of $\PP^1$ that give rise to a special subvariety of the moduli space
	of principally polarised abelian varieties.
	Denoting these families by $\M{1}$--$\M{21}$,
	more recent work of Li, Mantovan, Pries, and Tang
	considers five of these families 
	and shows that they contain supersingular curves as follows
	(and as left open in \cite[Theorem~1.1]{lmpt19}).
	\begin{theorem}[{\cite[Theorem~7.1]{li-mantovan-pries-tang}}]\label{thm:one}
		In each of the following families
		there exists a supersingular smooth curve of genus $g$ defined
		over $\overline{\FF_p}$
		for all sufficiently large primes~$p$
		that satisfy the given condition:
		\begin{enumerate}[1.]
			\item in  $\M{6}$ with $g=3$, assuming $p\equiv 2\pmod{3}$;
			\item in  $\M{8}$ with $g=3$, assuming $p\equiv 3\pmod{4}$;
			\item in  $\M{10}$ with $g=4$, assuming $p\equiv 2\pmod{3}$;
			\item in  $\M{14}$ with $g=4$, assuming $p\equiv 5\pmod{6}$;  and
			\item in  $\M{16}$ with $g=6$, assuming $p\equiv 2$, $3$ or $4 \pmod{5}$.
		\end{enumerate}
	\end{theorem}

By taking reductions modulo primes of known complex multiplication
curves, we prove the following more explicit version of cases 1, 2, and 5.
\begin{theorem}\label{thm:two}
	In each of the following families there exists a supersingular smooth
	curve of genus $g$ defined over $\FF_p$ for all primes $p$ that satisfy the given condition:
	\begin{enumerate}[1.]
		\item in  $\M{6}$ with $g=3$, assuming $p\equiv 2\pmod{3}$;
		\item in  $\M{8}$ with $g=3$, assuming $p\equiv 3\pmod{4}$;
		and
		\addtocounter{enumi}{1}
		\addtocounter{enumi}{1}
		\item  in  $\M{16}$ with $g=6$, assuming $p\equiv 2$, $3$ or $4 \pmod{5}$ and Conjecture \ref{conj:somozacurve}.
	\end{enumerate}
The curves are given explicitly in Theorems \ref{thm:M6}, \ref{thm:M8},
and \ref{thm:M16} respectively.
\end{theorem}

The improvement that Theorem~\ref{thm:two} provides over Theorem~\ref{thm:one} is as follows:
\begin{enumerate}[(a)]
	\item we remove the condition that $p$ be sufficiently large,
	\item we give explicit equations of the curves,
	\item we give curves that are defined not only over
	$\overline{\FF_p}$, but over~$\FF_p$.
\end{enumerate}

\section{\texorpdfstring{Supersingular Picard curves in $\M{6}$}{Supersingular Picard curves in \M{6}}}

The family $\M{6}$ consists of the smooth algebraic
curves that have a model of the form $y^3 = f(x)$
for a separable polynomial $f$ of degree~$4$.
This is a special family of smooth plane quartic curves
known as Picard curves.
It features in the LMFDB~\cite{lmfdb}
as \href{https://www.lmfdb.org/HigherGenus/C/Aut/3.3-1.0.3-3-3-3-3.1}{3.3-1.0.3-3-3-3-3.1}
and
\href{https://www.lmfdb.org/HigherGenus/C/Aut/3.3-1.0.3-3-3-3-3.2}{3.3-1.0.3-3-3-3-3.2}.

Koike-Weng~\cite{koike-weng} and Lario-Somoza-Vincent~\cite{LarioSomoza}
computed many
examples of complex multiplication (CM) curves in this family.
However, in order to prove Theorem~\ref{thm:one}.1
we can make do with one curve $$ C : y^3 = x^4 - x.$$
Note that it has an automorphism
$[\zeta_9] : (x,y)\mapsto (\zeta_9^3 x, \zeta_9 y)$
over $K = \QQ(\zeta_9)$.
This induces an automorphism of the Jacobian $J = \mathrm{Jac}(C)$ over $K$
and in fact an embedding $\ZZ[\zeta_9] \rightarrow \mathrm{End}(J_K)$,
proving that this curve indeed has CM.

For curves of genus $3$ with CM by a cyclic sextic field, Proposition~4.1
of K{\i}l{\i}{\c{c}}er-Labrande-Lercier-Ritzenthaler-Sijsling-Streng~\cite{ateam} tells us exactly when the reduction modulo $p$
is supersingular.
We get the following result.
\begin{theorem}\label{thm:M6}
	For all prime numbers $p\equiv 2\pmod{3}$
	the plane projective curve over $\FF_p$ given by
	$$y^3 = x^4 - x$$
	is smooth and supersingular in the family $\M{6}$.
\end{theorem}
\begin{proof}
The given model of $C$ has good reduction modulo every
prime $p\not=3$. Then
Proposition~4.1 of \cite{ateam} shows that
the reduction of $J$ modulo $p$ is supersingular if and only if $p\mathcal{O}_K$ has exactly $1$ or $3$ prime factors.
This is the case exactly when $p\not\equiv 1\pmod{3}$.
\end{proof}
For a self-contained proof
we can use the following lemma
instead of \cite[Proposition~4.1]{ateam}.
Note that $(p\bmod 9)$ has even order in $(\ZZ/9\ZZ)^\times$
exactly if $p\equiv 2\pmod{3}$.

\begin{lemma}\label{lem:insteadofkllrss}
	Let $K$ be a CM field that is cyclic of some degree $2g$ over $\QQ$.
	Let $A$ be an abelian variety of dimension $g$ over a number field $k$
	with endomorphism ring isomorphic to~$\mathcal{O}_K$.
	
	Let $p$ be a prime number that is unramified
	in $K$ and such that $\mathrm{Frob}_p
	\in \mathrm{Gal}(K/\QQ)$ has even order.
	Let $\mathfrak{P}\mid (p)$ be a prime of $k$ of
	good reduction for~$A$.
	
	Then the reduction $(A\bmod \mathfrak{P})$
	is supersingular.
\end{lemma}
\begin{proof}
	Without loss of generality (that is, by extending~$k$) all endomorphisms
	of $A$ are defined over~$k$.
	Then the Shimura-Taniyama formula \cite[Theorem 1($\pi2$) in Section 13]{shimura-taniyama}
	tells us that the Frobenius endomorphism of $(A\bmod \mathfrak{P})$ is
	an element $\pi \in \mathcal{O}_K$.
	
	We claim now that for every Galois conjugate $\pi'$ of $\pi$ 
	we have $\ord_{\mathfrak{P}}(\pi') = \frac{1}{2}\ord_{\mathfrak{P}}(p)$.
	Assuming the claim, we get for every $i$ that the coefficient $a_{2g-i}$ of $X^{2g-i}$
	in the characteristic polynomial of $\pi$
	satisfies
	$\ord_{\mathfrak{P}}(a_{2g-j}) \geq i\cdot \frac{1}{2}\ord_{\mathfrak{P}}(p)$.
	The Newton polygon therefore is a straight line segment from $(0,0)$
	to $(2g,g)$ of slope $\frac{1}{2}$, hence $A$ is supersingular.
	
	Now it remains only to prove the claim, which we do as follows.
	By assumption, the decomposition group of $p$ in $K/\QQ$
	contains an element of order $2$ in~$\mathrm{Gal}(K/\QQ)$,
	and since $\mathrm{Gal}(K/\QQ)$ is cylic, the only such element is complex conjugation.
	In particular, we get for every $\alpha\in K$:
	$$\ord_{\gp}(\overline{\alpha}) = \ord_{\overline{\gp}}(\alpha) = \ord_{\gp}(\alpha).$$
	
	We also have $\pi'\overline{\pi'} = p$ for every conjugate $\pi'$ of~$\pi$,
	hence
	$$\ord_{\gp}(p) = \ord_{\gp}(\pi') + \ord_{\gp}(\overline{\pi'})
	= 2\ord_{\gp}(\pi'),$$
	which proves the claim and finishes the proof of the lemma.
\end{proof}

\section{\texorpdfstring{Supersingular hyperelliptic curves in $\M{8}$}{Supersingular hyperelliptic curves in \M{8}}}

The family $\M{8}$ consists of the quadruple covers
of the projective $(U,V)$-line of the form
$$ C : Y^4 = F_2(U,V)\  F_3(U,V)^2,$$
where the $F_i$ are coprime homogeneous
separable polynomials
of degree~$i$.

Suppose that $F_2$ splits. A change
of variables then gives $F_2 = UV$, or in affine
coordinates
$$ C : y^4 = uf(u)^2$$
for a separable polynomial $f$ of degree $3$ with $f(0)\not=0$.

Using a further change of variables puts this curve
in hyperelliptic Weierstrass form.
Indeed, let $x$ be the element $x = y^2 / f(u)$ in the function field $k(C)$ of~$C$.
Then $k(C)$ is generated by $x$ and $y$ because
we have $u = x^2$.
Moreover, these new coordinates satisfy
$$ y^2 = xf(u) = xf(x^2),$$
hence we write our curve also as
$$ C : y^2 = xf(x^2).$$

Weng~\cite{weng-g3} gave a construction for CM curves
of this latter Weierstrass form.
For example, she constructed the curve 
$$y^2 = u^7 + 6u^5 + 9u^3 + u,$$
which has complex multiplication by
$\QQ(\zeta_9)^+(i)$, as
was later proven by Costa-Mascot-Sijsling-Voight~\cite{cmsv, cmsv-software}.
We get the following result.

\begin{theorem}\label{thm:M8}
For all prime numbers $p\equiv 3\ \mathrm{mod}\ 4$
except $p=3$,
the curve
$$y^2 = xf(x^2)\quad\mbox{or equivalently}\quad
y^4 = uf(u)^2,$$
with
$$f(u) = u^3 + 6u^2 + 9u + 1$$
is a smooth supersingular curve of genus $3$ defined over $\FF_p$ 
in the family $\M{8}$.
For the prime $p=3$, the curve
with $f(u) = u^3 + 7u^2 + 14u + 7$
has the same property.
\end{theorem}
\begin{proof}
	For $p=3$ this is one curve, so a direct verification
	by counting points proves the result.
	We do this using SageMath~\cite{sage}, see
	the accompanying code~\cite{code}.
	
	For $p\not=3$ the proof is exactly the same
	as for Theorem~\ref{thm:M6}
	(using either \cite[Proposition~4.1]{ateam} or Lemma~\ref{lem:insteadofkllrss}),
	but with the field $\QQ(\zeta_9)^+(i)$.
\end{proof}
\begin{remark}
	The polynomial $u^3 + 7u^2 + 14u + 7$
	again comes from Weng~\cite{weng-g2}.
	Since we only use it over $\FF_3$, we could have equivalently written
	$u^3 + u^2 - u + 1$.
\end{remark}

\section{\texorpdfstring{Supersingular cyclic plane quintics in $\M{16}$}{Supersingular cyclic plane quintics in \M{16}}}

The family $\M{16}$ consists of the smooth plane curves
$y^5 = f(x)$ for a separable polynomial $f$ of
degree $4$ or~$5$.
These curves were
dubbed \emph{cyclic plane quintics} or \emph{CPQ curves}
by Somoza
who constructed
in~\cite[Section~2.3]{Somozathesis} 
the curve
\begin{equation}\label{eq:somozacurve}
	C : y^5 = x^4 - 24 x^3 + 3 x^2 + x
	\end{equation}
and conjectured the following.
\begin{conjecture}\label{conj:somozacurve}
	The Jacobian $J$ of the curve $C$ of \eqref{eq:somozacurve} has
	$\mathrm{End}(J_{\overline{\QQ}}) \cong \mathcal{O}_K$,
	where
	$K =\QQ(\zeta_9)^+(\zeta_5)$.
\end{conjecture}
\begin{remark}
I have no doubt about the validity of the conjecture,
though rigourous verification seems to be computationally
just out of reach using the methods of~\cite{cmsv},
see~\cite{costaticket}.

We have the following evidence for the conjecture. First of all,
it is stated in \cite[Table~4.2]{Somozathesis} that there exists
a curve
of the form $y^2 = f(x)$ with $f(x) \in \QQ[x]$
such that its Jacobian has CM over $\overline{\QQ}$ by~$K$.
The method for computing this curve is high-precision
numerical approximation, hence $C$
is numerically extremely close to a curve that does have the
correct~CM.

Next, the curve $C$ that Somoza obtained in this way has
discriminant $3^{10}$. Such a smooth number seems unlikely
to appear by chance.

Finally, we verified that $C$ is supersingular modulo $2$, $7$,
$13$, and $17$, consistently with Theorem~\ref{thm:M16} below
(see~\cite{code}).
This again seems unlikely to be by chance.
\end{remark}

\begin{theorem}\label{thm:M16}
	Assume Conjecture~\ref{conj:somozacurve}.
For all prime numbers $p\equiv 2$, $3$, or $4\pmod{5}$ except $p=3$
the curve over $\FF_p$ given by \eqref{eq:somozacurve}
is supersingular.

	Over $\FF_3$, the curve given by $y^5 = x^4 - 7x^2 + 7x$ is supersingular.
\end{theorem}
\begin{proof}
	Again we check the curve for $p = 3$ separately.
	For the other primes, the equation \eqref{eq:somozacurve}
	has good reduction.
	Assume Conjecture~\ref{conj:somozacurve}
	and apply Lemma~\ref{lem:insteadofkllrss}.

	The primes $p\not=3$ that are $2$, $3$, or $4\pmod{5}$
	are unramified in $K/\QQ$ and have order $2$ or $4$ in $(\ZZ/5\ZZ)^\times$.
	In particular, the corresponding Frobenius automorphism
	of $\QQ(\zeta_5)$ has even order, and so does that of~$K$.
	Therefore, Lemma~\ref{lem:insteadofkllrss} gives the result.
\end{proof}

\begin{remark}
	In this case, Proposition~4.1 of \cite{ateam} does not apply directly,
	as it is stated only for abelian threefolds.
	The method of proof applies to more general abelian varieties.
	We condensed it to what we needed and then generalised this,
	which resulted in Lemma~\ref{lem:insteadofkllrss}.
\end{remark}

\section{Final notes}

\subsection{Other families}

We restricted to the three families $\M{6}$, $\M{8}$, and $\M{16}$
because the literature
contains algorithms for reconstructing curves in these families from their period matrices
(and in fact even contains the required example curves).
Completely new reconstruction algorithms are
far beyond the scope of this paper and
are in fact a project in development by the author.

\subsection{Databases of curves, and verification}

The curves that we needed came from three different sources in the literature.
More importantly, the proof that they were correct came from yet two other sources
or (in the case of $\M{16}$) does not even exist yet.
A database with CM curves and certificates of their Jacobians' endomorphisms
would be very useful.

\bibliographystyle{amsplain}
	\bibliography{supersingular}

\end{document}